\documentclass{article}
\usepackage{amsmath, amsfonts, amsthm, amssymb}
\usepackage{graphicx}
\usepackage{float}
\usepackage{verbatim}

\numberwithin{equation}{section}
\allowdisplaybreaks

\newtheorem{thm}{Theorem}[section]
\newtheorem{lma}[thm]{Lemma}
\newtheorem{cor}[thm]{Corollary}
\newtheorem{defn}[thm]{Definition}

\newtheorem{prop}[thm]{Proposition}
\newtheorem{conj}[thm]{Conjecture}

\renewcommand{\geq}{\geqslant}
\renewcommand{\leq}{\leqslant}
\renewcommand{\H}{\text{H}}

\title{Micromeasure distributions and applications for \\ conformally generated fractals}

\author{Jonathan M. Fraser\\
\emph{School of Mathematics, The University of Manchester,}\\ \emph{Manchester, M13 9PL, UK}\\
\emph{Email: jonathan.fraser@manchester.ac.uk} \\ \\
 Mark Pollicott\\
\emph{Mathematics Institute, Zeeman Building,}\\ \emph{University of Warwick, Coventry, CV4 7AL, UK}\\
\emph{Email: mpollic@maths.warwick.ac.uk}}

\begin{document}
\maketitle

\begin{abstract}
We study the scaling scenery of Gibbs measures for subshifts of finite type on self-conformal fractals and applications to Falconer's distance set problem and dimensions of projections.  Our analysis includes hyperbolic Julia sets, limit sets of Schottky groups and graph-directed self-similar sets.
\\


\end{abstract}

\section{Introduction}

This article concerns Gibbs measures supported on subshifts of finite type and corresponding classes of fractals defined via iterated function systems consisting of conformal maps, for example hyperbolic Julia sets, limit sets of Schottky groups and graph-directed self-similar sets.  Specifically, we are interested in understanding the scaling scenery for these measures, which can be described by Furstenberg's notion of CP-chains, and in applications in the geometric setting to the dimension theory of projections and distance sets. As such we build on recent and significant developments in the area due to Hochman and Shmerkin \cite{HochmanShmerkin}.  Studying the process of zooming in on a fractal set or measure is very much in vogue at the moment and is proving useful in many contexts.  We note that this kind of problem has been considered for certain conformally generated fractals before.  In particular, we mention the papers \cite{bedfordfisher1, bedfordfisher2, patz}, which share some of the spirit of this article.  These papers were largely concerned with a detailed analysis of the scaling scenery, whereas we place more emphasis on geometric applications.  Our applications include: resolution of Falconer's distance set problem for Julia sets with hyperbolic dimension strictly greater than one and an extension of Hochman and Shmerkin's optimal projection theorem for self-similar sets to the graph-directed setting where the action induced by the defining mappings on the Grassmannian manifold need not be a group action.

\subsection{Micromeasures and CP-distributions}

Furstenberg \cite{Furstenberg} introduced the notion of a \emph{CP-chain}  (\emph{conditional probability chain}) to capture the dynamics of the process of zooming in on a fractal measure, although many of the ideas are already present in his 1970 article \cite{Furstenberg60s}.  We will not use  CP-chains directly and so refer the reader to the papers \cite{Furstenberg, Hochman, HochmanShmerkin, kaenmaki} for a more in-depth account.  However, we will rely heavily on the theory of CP-chains developed by Hochman and Shmerkin \cite{HochmanShmerkin} and so will recall various results from that work as we go.  Write $\mathcal{P}(K)$ for the space of Borel probability measures supported on a compact metric space $K$ and $\text{supp} (\mu) \subseteq K$ for the support of $\mu \in \mathcal{P}(K)$. A \textit{distribution} is a member of $\mathcal{P}(\mathcal{P}(K))$ where $\mathcal{P}(K)$ has been metrized in a way compatible with the topology of weak-$*$ convergence, for example with the Levy-Prokhorov or Wasserstein metric.
\\ \\
Let $b \in \mathbb{N}$ with $b\geq 2$ and let $\mathcal{E}_b$ be the collection of all half open $b$-adic boxes contained in $[0,1)^d$ oriented with the coordinate axes which are the product of half open $b$-adic intervals of the same generation. If $x \in [0,1)^d$, write $\Delta_b^n(x)$ for the unique $n$th generation box in $\mathcal{E}_b$ containing $x$, i.e. $\Delta_b^n(x)$ is a product of $d$ half open intervals of length $b^{-n}$. For $B \in \mathcal{E}_b$, let $T_B : \mathbb{R}^d \to \mathbb{R}^d$ be the unique rotation and reflection free similarity that maps $B$ onto $[0,1)^d$.  If $\mu \in \mathcal{P}([0,1]^d) $ and $\mu(B) > 0$, write
\[
\mu^B \ = \ \frac{1}{\mu(B)} \mu|_B \circ T_B^{-1} \ \in \  \mathcal{P}([0,1]^d).
\]
In practise what one is often interested in are the ($b$-adic) minimeasures $\mu^{\Delta^k_b(x)}$ at a generic point $x$ in the support of $\mu$ and the weak limits of such measures which are called ($b$-adic) micromeasures (at $x$).  We denote the set of all minimeasures of $\mu$ by $\text{Mini}(\mu)$ and the set of all micromeasures of $\mu$ by $\text{Micro}(\mu)$.  In general, (the measure component of) a CP-chain is a special type of distribution $Q \in \mathcal{P}(\mathcal{P}([0,1]^d))$ and of central importance is the notion of a measure $\mu$ \emph{generating} an (ergodic) CP-chain, see \cite[Section 7]{HochmanShmerkin}. Indeed a lot of the subsequent applications will apply to `measures which generate ergodic CP-chains'.  The definitions involved are fairly technical but can often be sidelined in practise due to the following trick of Hochman-Shmerkin.  Theorem 7.10 in \cite{HochmanShmerkin} implies that for any $\mu \in \mathcal{P}([0,1]^d)$, there exists an ergodic CP-chain whose measure component $Q$ is supported on $\text{Micro}(\mu)$ and has dimension at least $\dim_\H \mu$.  The `dimension' of a CP-chain is the average of the dimensions of micromeasures with the respect to the measure component of the chain,  i.e.
\[
\int \dim_{\H} \nu \, d Q(\nu),
\]
but for an ergodic CP-chain the micromeasures are almost surely exact dimensional with a common `exact dimension', \cite[Lemma 7.9]{HochmanShmerkin}. Theorem 7.7 in \cite{HochmanShmerkin} tells us that $Q$-almost all $\nu \in \text{Micro}(\mu)$ generate this CP-chain.  This means that if $\mu$ is sufficiently regular that all of its micromeasures are `geometrically similar' to $\mu$ itself, then applying the machinery of CP-chains to the micromeasures is sufficient to obtain geometric results concerning $\mu$.  This is a central theme of this paper.

\subsection{Applications to projections and distance sets}

Relating the dimension and measure of orthogonal projections of subsets of Euclidean space to the dimension and measure of the original set is a classical problem in geometric measure theory, see the recent survey \cite{FalconerFraserJin}.  Throughout this article we will be concerned with the Hausdorff dimension $\dim_\H$ of sets and the (lower) Hausdorff dimension of measures, defined by $\dim_\H \mu = \inf\{ \dim_\H E : \mu(E)>0\}$.  In particular, the Hausdorff dimension of a measure is at most the Hausdorff dimension of its support.  We refer the reader to the books \cite{Falconer, Mattila} for more details on the dimension theory of sets and measures.  The seminal results of Marstrand, Kaufman and Mattila have established that the dimension is `almost surely what it should be' in the following sense, see \cite[Chapter 9]{Mattila}.

\begin{thm}[Marstrand-Kaufman-Mattila]
Let $K \subset \mathbb{R}^d$ be compact and let $k \in \{1, \dots, d-1\}$.  Then for almost all orthogonal projections $\pi \in \Pi_{d,k}$, we have
\[
\dim_{\text{\emph{H}}} \pi K \ = \ \min \{ k,  \dim_{\text{\emph{H}}} K \},
\]
where $\Pi_{d,k}$ is the Grassmannian manifold consisting of all orthogonal projections from $\mathbb{R}^d$ to $\mathbb{R}^k$ equipped with the natural measure.
\end{thm}

We note that there are analogues of this theorem for projections of \emph{measures}, see \cite[Section 10]{FalconerFraserJin}.  Recently many people have been concerned with strengthening the above result in specific settings with the philosophy that the only exceptions should be the evident ones.  One of the major advances on this front was due to applications of the CP-chain machinery by Hochman and Shmerkin \cite{HochmanShmerkin}.
\begin{thm}[See Theorem 8.2 of \cite{HochmanShmerkin}] \label{EEE}
Suppose $\mu \in \mathcal{P}([0,1]^d)$ generates an ergodic CP-chain and let $k \in \mathbb{N}$ and $\varepsilon>0$.  Then there exists an open dense set $\mathcal{U}_\varepsilon \subset \Pi_{d,k}$ (which is also of full measure) such that for all $\pi \in \mathcal{U}_\varepsilon$
\[
\dim_{\text{\emph{H}}} \pi \mu > \min\{k, \dim_{\text{\emph{H}}}\mu\} - \varepsilon.
\]
\end{thm}
A key application of this result is to obtain `all projection' type results in certain situations. Indeed, if one can show that $\dim_\H \pi \mu $ is invariant under some minimal action on $\Pi_{d,k}$, then this forces it to be constantly equal to the value predicted by Marstand-Kaufman-Mattila.  Hochman and Shmerkin also obtain a nonlinear projection theorem, which is a testament to the robustness of the CP-chain approach to projection type problems.  

\begin{thm}[See Proposition 8.4 of \cite{HochmanShmerkin}] \label{C1images}
Suppose $\mu \in \mathcal{P}([0,1]^d)$ generates an ergodic CP-chain. Let $\pi \in \Pi_{d,k}$ and $\varepsilon>0$. Then there exists $\delta>0$ such that for all $C^1$ maps $g : [0,1]^d \to \mathbb{R}^k$ with
$$\sup_{x \in \text{\emph{supp}} (\mu)} \|D_x g - \pi\| < \delta,$$
we have
$$\dim_\text{\emph{H}} g \mu > \dim_\text{\emph{H}} \pi \mu - \varepsilon.$$
\end{thm}

It was a recent innovation of Orponen \cite{Orponen} that the work by Hochman and Shmerkin on $C^1$ images could be adapted to obtain information about the dimension of \emph{distance sets}.  Given a set $K \subset \mathbb{R}^d$, the distance set of $K$ is defined by
\[
D(K) \ = \ \big\{ \lvert x-y\rvert : x,y \in K \big\}.
\]
Of particular interest is \emph{Falconer's distance set conjecture}, originating with the paper \cite{distancesets}, which generally tries to relate the dimension of $D(K)$ with the dimension of $K$.  One version of the conjecture is as follows:
\begin{conj} \label{distanceconj}
Let $K \subseteq \mathbb{R}^d$ be analytic.  If $\dim_{\text{\emph{H}}} K \geq d/2$, then $\dim_{\text{\emph{H}}} D(K) =1$.
\end{conj}
There have been numerous partial results in a variety of directions but the full conjecture still remains a major open problem in geometric measure theory, see for example \cite{erdogan,bourgain,Orponen} and the references therein.  Orponen \cite{Orponen} considered the distance set problem for self-similar sets, but a more general result was proved by Ferguson, Fraser and Sahlsten building on the idea of Orponen, which we now state.

\begin{thm}[Theorem 1.7 of \cite{FergusonFraserSahlsten}] \label{FFSdist}
Let $\mu$ be a measure on $\mathbb{R}^2$ which generates an ergodic CP-chain and satisfies $\mathcal{H}^1\big(\text{\emph{supp}}(\mu)\big) >0$. Then
\[
\dim_{\text{\emph{H}}} D\big(\text{\emph{supp}}(\mu) \big) \geq \min \{ 1, \ \dim_{\text{\emph{H}}} \mu \}.
\]
\end{thm}
This theorem was applied in \cite{FergusonFraserSahlsten} to prove Conjecture \ref{distanceconj} for certain planar self-affine carpets.

\subsection{Our setting}

\subsubsection{Invariant measures on subshifts}

Let $\mathcal{I} = \{0, \dots, M-1\}$ be a finite alphabet, let $\Sigma  = \mathcal{I}^{\mathbb{N}}$ and $\sigma : \Sigma \to \Sigma$ be the one-sided left shift.  Write $i \in \mathcal{I}$, $\textbf{\emph{i}} = (i_0 , \dots, i_{k-1}) \in \mathcal{I}^k$, $\alpha = (\alpha_0, \alpha_1, \dots) \in \Sigma$ and $\alpha\vert_k = (\alpha_0, \dots, \alpha_{k-1}) \in \mathcal{I}^k$ for the restriction of $\alpha$ to its first $k$ coordinates.  We equip $\Sigma$ with the standard metric defined by
\[
d(\alpha,\beta) = 2^{-n(\alpha,\beta)}
\]
for $\alpha \neq \beta$, where $n(\alpha,\beta) = \max\{ n \in \mathbb{N} : \alpha\vert_n = \beta\vert_n\}$.  Any closed $\sigma$-invariant set $\Lambda \subseteq \Sigma$ is called a \emph{subshift}.  Among the most important subshifts are \emph{subshifts of finite type} which are defined as follows.  Let $A$ be an $M \times M$ \emph{transition matrix} indexed by $\mathcal{I} \times \mathcal{I}$ with entries in $\{0,1\}$.  We define the subshift of finite type corresponding to $A$ by
\[
\Sigma_A \ = \ \Big\{\alpha = (\alpha_0 \alpha_1 \dots ) \in \Sigma : A_{\alpha_i, \alpha_{i+1}} = 1 \text{ for all } i =0, 1, \dots  \Big\}.
\]
We say (the shift on) $\Sigma_A$ is \emph{transitive} if the matrix $A$ is irreducible, which means that for all pairs $i,j \in \mathcal{I}$, there exists $n \in \mathbb{N}$ such that $(A^n)_{i,j} >0$.  We say (the shift on) $\Sigma_A$ is \emph{mixing} if the matrix $A$ is aperiodic, which means that there exists $n \in \mathbb{N}$ such that $(A^n)_{i,j} >0$ for all pairs $i,j \in \mathcal{I}$ simultaneously.  For $\alpha \in \Sigma_A$ and $n \in \mathbb{N}$ write
\[
[\alpha\vert_n] = \{ \beta \in \Sigma_A : \beta\vert_n = \alpha\vert_n \}
\]
to denote the cylinder corresponding to $\alpha$ at depth $n$.  The cylinders generate the Borel $\sigma$-algebra for $( \Sigma_A, d)$.  An important class of measures naturally supported on subshifts of finite type are \emph{Gibbs measures}, see \cite{Bowen}.  Let $\phi:\Sigma_A \to \mathbb{R}$ be a continuous potential and define the $n$th variation of $\phi$ as
\[
\text{var}_n(\phi) \ = \ \sup_{\alpha, \beta \in \Sigma_A} \Big\{ \lvert \phi(\alpha)  - \phi(\beta) \rvert : \alpha\vert_n = \beta\vert_n  \Big\}.
\]
It is clear that $\text{var}_n(\phi) $ forms a decreasing sequence and that $\text{var}_n(\phi)  \to 0$ is equivalent to $\phi$ being continuous. We will assume throughout that $\phi$ has \emph{summable variations}, i.e.
\[
\sum_{l=0}^{\infty} \text{var}_{l}(\phi) < \infty.
\]
For $n \in \mathbb{N}$, let
\[
 \phi^n(\alpha)  \ = \ \sum_{l=0}^{n-1} \phi\big(\sigma^{l}(\alpha) \big).
\]
A measure $\mu \in \mathcal{P}(\Sigma_A)$ is called a \emph{Gibbs measure} for $\phi$ if there exists constants $C_1, C_2$ such that
\begin{equation} \label{gibbsbowen}
C_1\ \leq \ \frac{\mu\big( [\alpha\vert_n]\big)}{\exp( \phi^n(\alpha) - n P(\phi))} \ \leq \ C_2
\end{equation}
for all $\alpha \in \Sigma_A$ and all $n \in \mathbb{N}$, where $P(\phi)$ is the pressure of $\phi$.  If $\phi$ has summable variation and $\Sigma_A$ is mixing, then it has a unique $\sigma$-invariant Gibbs measure, but we do not necessarily assume our Gibbs measures are invariant in this paper.

\subsubsection{Conformally generated sets and measures}

Let $\{S_i\}_{i \in \mathcal{I}}$ be a finite collection of contracting self-maps on a compact metric space $X$ (which later we will assume to be a compact subset of either $\mathbb{C}$ or $\mathbb{R}^d$) indexed by the alphabet $\mathcal{I}$. For $\alpha \in \Sigma$ and $k \in \mathbb{N}$, write
\[
S_{\alpha\vert_k} = S_{\alpha_0} \circ \cdots \circ S_{\alpha_{k-1}},
\]
and
\[
\text{Lip}^+(S_{\alpha\vert_k}) \ = \  \sup_{x,y \in X} \frac{\lvert S_{\alpha\vert_k}(x)-S_{\alpha\vert_k}(y) \rvert}{\lvert x-y \rvert}
\]
and
\[
\text{Lip}^-(S_{\alpha\vert_k}) \ = \  \inf_{x,y \in X}  \frac{\lvert S_{\alpha\vert_k}(x)-S_{\alpha\vert_k}(y) \rvert}{\lvert x-y \rvert}
\]
for the upper and lower Lipschitz constants respectively.  Define the natural coding map $\Pi: \Sigma \to X$ by
\[
\Pi(\alpha) = \bigcap_{k=1}^\infty S_{\alpha\vert_k} (X)
\]
and let $F = \Pi(\Sigma)$ and $F_A = \Pi(\Sigma_A)$. Sometimes we will be interested in the first level cylinders of $F_A$, so for $i \in \mathcal{I}$ let $F_A^i =\Pi(\Sigma_A \cap [i])$.  Let $\mu_\text{sym}$ be a Gibbs measure for a potential with summable variations supported on a subshift of finite type and let $\mu = \mu_{\text{sym}} \circ \Pi^{-1}$.  We say that $\mu$ is a Gibbs measure for $F_A$ and observe that $\text{supp}(\mu) = F_A$. We will sometimes be interested in $\mu$ restricted to first level cylinders and so for $i \in \mathcal{I}$, write $\mu_i$ for $\mu$ restricted to $F_A^i$.  The sets $F$ and $F_A$ and measures $\mu$ are our main object of study in this article.  Note that for all $i \in \mathcal{I}$, $\dim_\H \mu_i = \dim_\H \mu$ and $\dim_\H F_A^i = \dim_\H F_A$.  Some of our results will require a certain `separation condition', which we now state. We say two finite words are \emph{incomparable} if neither is word is a subword of the other.

\begin{defn}[Strong separation property]  The set $F_A$ satisfies the \emph{strong separation property} if for all $\alpha, \beta \in \Sigma_A$ and $k,l \in \mathbb{N}$ such that $\alpha\vert_k$ and $\beta\vert_l $ are incomparable, we have $ S_{\alpha\vert_k}(X)  \cap S_{\beta\vert_l}(X)= \emptyset$.   
\end{defn}

We will now specialise to two particular settings:
\\ \\
(1) We will say the system $\{S_i\}_{i \in \mathcal{I}}$ is a \emph{conformal system} if the compact  metric space $X$ on which the maps $S_i$ act is the closure of some open simply connected region $U \subseteq \mathbb{C}$ and each map $S_i$ is conformal on $U$.  We assume for convenience that $U = \{z=x+iy \in \mathbb{C} : x,y \in (0,1)\}$, which we may do by applying the Riemann Mapping Theorem.  Recall that such a map is conformal if and only if it is holomorphic (equivalently analytic) on $U$ with non-vanishing derivative.  Two simple consequences of this assumption are that the Jacobian derivative $D_xS_i$ of $S_i$ exists at every $x \in U$ and is equal to a scalar times an orthogonal matrix and that there exists a uniform constant $L\geq 1$ such that for all $\alpha \in \Sigma$ and $k \in \mathbb{N}$
\[
\frac{\text{Lip}^+(S_{\alpha\vert_k})}{\text{Lip}^-(S_{\alpha\vert_k})} \ \leq \ L.
\]
This last phenomenon is often referred to as \emph{bounded distortion}. Two key examples of sets which can be realised by conformal systems are \emph{limit sets of Schottky groups} and \emph{hyperbolic Julia sets}.  (Classical) Schottky groups are a special type of Kleinian group generated by reflections in a collection of disjoint circles in some region of the complex plane.  As such the limit set can be realised as $F_A$ for a conformal system and a subshift of finite type with matrix $A$ having 1s everywhere apart from the main diagonal where it has 0s due to the fact that the part of the limit set inside one particular circle will not contain a copy of itself.  Julia sets $J$ on the other hand are dynamical repellers for complex rational maps $f$, and if $J$ lies in a bounded region of the complex plane and $f$ is strictly expanding on $J$, then $J$ can be viewed as the self-conformal attractor of the iterated function system formed by the inverse branches of $f$ defined on a neighbourhood of $J$.  Such Julia sets can thus be realised as $F$ in our setting.
\begin{figure}[H] 
	\centering
	\includegraphics[width=125mm]{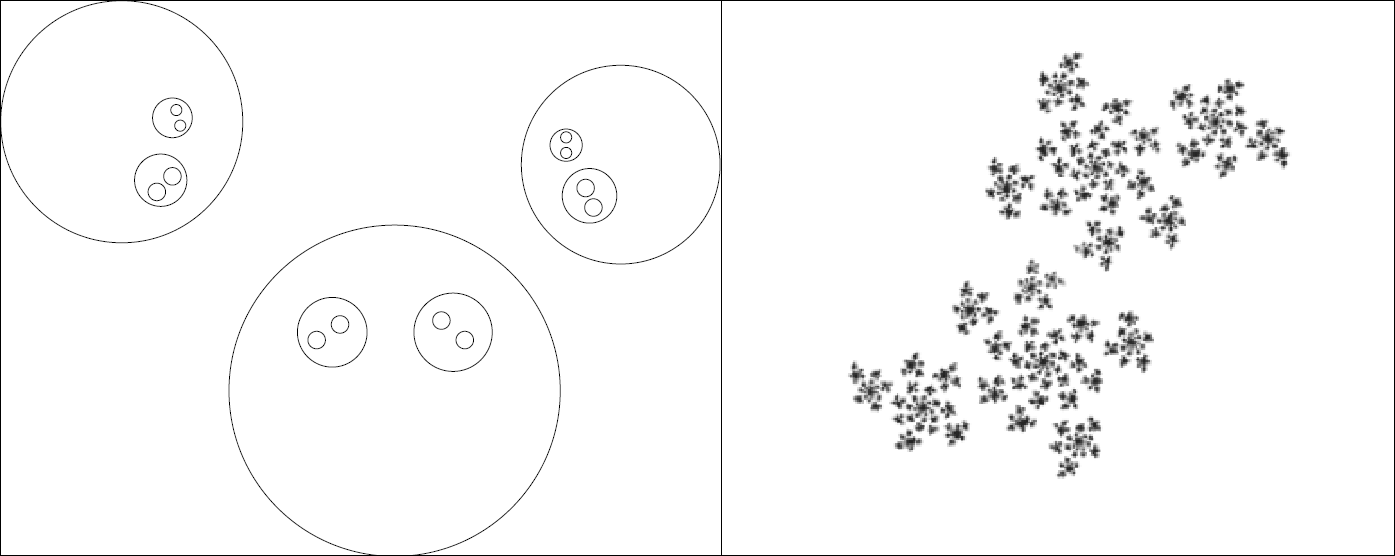}
\caption{Left: circles generating a Schottky group and the construction of the limit set.  Right: a self-conformal Julia set.}
\end{figure}
(2) We will say the system $\{S_i\}_{i \in \mathcal{I}}$ is a \emph{system of similarities} if the compact metric space $X$ on which the maps $S_i$ act is a compact subset of Euclidean space, which we assume is equal to $[0,1]^d$ for some $d \in \mathbb{N}$, and each map $S_i$ is a similarity.  Two key examples of sets which can be realised by  systems of similarities are \emph{self-similar sets} and \emph{graph-directed self-similar sets}, see \cite[Chapter 9]{Falconer}.  Indeed, the set $F$ corresponding to the full shift is a self-similar set and for a transitive subshift of finite type $\Sigma_A$ the first level cylinders $F_A^i$ ($i \in \mathcal{I}$) of $F_A$ form a family of graph-directed self-similar sets and every such family can be realised in this way, see \cite[Propositions 2.5-2.6]{FarkasFraser} for example.

\section{Results}

\subsection{Scenery for Gibbs measures on conformally generated fractals}

 The results in this section aim to provide a link between the Gibbs measures we study in this paper, their micromeasures and their micromeasure distributions.  This allows us to apply the machinery of CP-chains to conformally generated fractal sets and measures.

\begin{thm} \label{symbtogeom}
Consider a conformal system and a subshift of finite type satisfying the strong separation property and let $\mu$ be a Gibbs measure for $F_A$.  Then, for all $\nu \in  \text{Mini}(\mu) \cup \text{Micro}(\mu)$ which are not supported on the boundary of the square $\overline{U}$, there exists a conformal map $S$ on $U$ and a measure $\mu_0 \equiv \mu_i$ for some $i \in \mathcal{I}$, such that $\nu(S(F_A))>0$ and
\[
\nu\vert_{S(F_A)} \ = \ \mu_0 \circ S^{-1}.
\]
\end{thm}

We will prove Theorem \ref{symbtogeom} in Section \ref{symbtogeomproof}.  A similar result was proved by Hochman and Shmerkin in the case of full shifts for systems on the unit interval where each map is $C^{1+\alpha}$ \cite[Proposition 11.7]{HochmanShmerkin}. The analogous result for systems of similarities is proved similarly and is stated without proof.

\begin{thm} \label{symbtogeom2}
Consider a system of similarities and a subshift of finite type $\Sigma_A$ satisfying the strong separation property  and let $\mu$ be a Gibbs measure for $F_A$.   Then, for all $\nu \in  \text{Mini}(\mu) \cup \text{Micro}(\mu)$ which is not supported on the boundary of the hypercube $[0,1]^d$, there exists a similarity map $S$ on $[0,1]^d$ and a measure $\mu_0 \equiv \mu_i$ for some $i \in \mathcal{I}$, such that $\nu(S(F_A))>0$ and
\[
\nu\vert_{S(F_A)} \ = \  \mu_0 \circ S^{-1}.
\]
\end{thm}

 Understanding the micromeasures allows us to prove the following results.

\begin{thm} \label{symbtogeomCP}
Consider a conformal system and a subshift of finite type $\Sigma_A$  satisfying the  strong separation property and let $\mu$ be a Gibbs measure for $F_A$.   Then there exists a conformal map $S$ on $U$ and a measure $\mu_0 \equiv \mu_i$ for some $i \in \mathcal{I}$, such that $\mu_0 \circ S^{-1}$ generates an ergodic CP-chain of dimension at least $\dim_\text{\emph{H}} \mu$.
\end{thm}

We will prove  Theorem \ref{symbtogeomCP} in Section \ref{symbtogeomCPproof}.   In some sense it is unsatisfying that we need to take a conformal image (by $S$) before we can generate a CP-chain.  This is essential however, as the following example demonstrates.   One dimensional Lebesgue measure on the upper half of the boundary of the unit circle in $\mathbb{C}$ is a Gibbs measure for a conformal system modelled by a full shift on two symbols.  The defining maps can be taken to be $z \mapsto \sqrt{z}$ and $z \mapsto i \sqrt z$ for example.  Even though this measure is very regular, it does not generate a CP-chain because, although the micromeasures at every $x$ are simply Lebesgue measure supported on a line segment, the line segments are at different angles corresponding to the slope of the tangent to the unit circle at that $x$.  As predicted by Theorem \ref{symbtogeomCP} there is a conformal image of $\mu$ which does generate an (ergodic) CP-chain and this is none-other than Lebesgue measure restricted to any line segment. Again, the analogous result for systems of similarities is proved similarly and is stated without proof.

\begin{thm} \label{symbtogeomCP2}
Consider a system of similarities and a subshift of finite type $\Sigma_A$ satisfying the strong separation property and let $\mu$ be a Gibbs measure for $F_A$.  Then there exists a similarity map $S$ on $[0,1]^d$ and a measure $\mu_0 \equiv \mu_i$ for some $i \in \mathcal{I}$, such that $\mu_0 \circ S^{-1}$ generates an ergodic CP-chain of dimension at least $\dim_\text{\emph{H}} \mu$.
\end{thm}

A similar result in the case of Bernoulli measures on full shifts can be found in \cite[Proposition 9.1]{HochmanShmerkin}.

\subsection{Geometric applications}

\subsubsection{Approximating overlapping systems from within}

If we have a Gibbs measure for a conformal system which does not satisfy the strong separation property, then analysing the scaling scenery and micromeasure structure can be complicated.  However, often one is only interested in studying the support of the measure, not the measure itself.  As such if one can `approximate the system from within' by finding a subsystem with sufficient separation and which approximates the Hausdorff dimension of the larger set to within any $\varepsilon$, then one can often get the desired results, even for systems with overlaps.  The key to doing this is the following proposition.

\begin{prop} \label{approxfromwithin}
Consider a conformal system or a system of similarities, let $\Sigma_A$ be a transitive subshift of finite type and let $\varepsilon>0$.  Then there exists a full shift $\Sigma_\varepsilon$ over an alphabet made up of a finite collection of restrictions of elements in $\Sigma_A$ such that
\[
\Pi(\Sigma_\varepsilon) \subseteq F_A,
\]
\[
\dim_{\text{\emph{H}}} \Pi(\Sigma_\varepsilon) \geq \dim_{\text{\emph{H}}} F_A - \varepsilon
\]
and such that the system corresponding to $\Sigma_\varepsilon$ satisfies the strong separation property.
\end{prop}

We will prove Proposition \ref{approxfromwithin} in Section \ref{approxfromwithinproof}.  Similar results have been proved before for overlapping \emph{self-similar} sets, which correspond to \emph{full shifts}, see \cite{Orponen, Farkas}.  The main difficulty in our generalisation was ensuring our subsystem remained inside the subshift of finite type, even if this is a strict subshift.

\subsubsection{Distance sets}

To prove the distance set conjecture for conformally generated fractals we rely on the method used in \cite{FergusonFraserSahlsten} to prove Theorem \ref{FFSdist}.  However, the result there is not quite strong enough to obtain the desired result for the nonlinear sets we consider here.  The reason for this is that Theorem \ref{symbtogeom} does not show that Gibbs measures supported on $F_A$ generate ergodic CP-chains, but rather a conformal image of them does.  This combined with Theorem \ref{FFSdist} would only yield the distance set conjecture for a particular conformal image of $F_A$, which is clearly unsatisfactory.  Thus we prove the following strengthening of Theorem \ref{FFSdist}. 

\begin{thm} \label{FFS+dist}
Let $\mu$ be a measure on $\mathbb{C}$ which generates an ergodic CP-chain and satisfies $\mathcal{H}^1\big(\text{\emph{supp}}(\mu)\big) >0$. Then
\[
\dim_{\text{\emph{H}}} D\big(S\big(\text{\emph{supp}}(\mu)\big) \big) \geq \min \{ 1, \ \dim_{\text{\emph{H}}} \mu \circ S^{-1} \} =  \min \{ 1, \ \dim_{\text{\emph{H}}} \mu \}
\]
for {\bf \emph{any}} conformal map $S$.
\end{thm}

We will prove Theorem \ref{FFS+dist} in Section \ref{FFS+distproof}. Our main result on the distance set problem is the following.

\begin{thm} \label{distancesetsconformal}
Consider a conformal system or a system of similarities in the plane.  Then for any transitive subshift of finite type $\Sigma_A$ such that $\dim_{\text{\emph{H}}} F_A >1$, Falconer's distance set conjecture holds, i.e. 
\[
\dim_{\text{\emph{H}}} D \big(F_A  \big) =1.
\]
If we further assume the strong separation property, then the assumption $\dim_{\text{\emph{H}}} F_A >1$ can be relaxed to $\dim_{\text{\emph{H}}} F_A \geq 1$.
\end{thm}

We will prove Theorem \ref{distancesetsconformal} in Section \ref{distancesetsconformalproof}.  We note that this distance set result applies in several concrete settings.  Most simply it proves the conjecture for graph-directed self-similar sets without assuming any separation properties, hyperbolic Julia sets and limit sets of Schottky groups.  However, it also applies more generally since if $E \subseteq F$, then $D(E) \subseteq D(F)$. In particular, our result proves the conjecture for general Julia sets with \emph{hyperbolic dimension} strictly larger than 1.  For example, Bara\'nski, Karpi\'nska and Zdunik showed that this is the case for meromorphic maps with logarithmic tracts \cite{baranski}. Recall that hyperbolic dimension is the supremum of the Hausdorff dimensions of compact hyperbolic subsets.   It is an important open problem to determine for which rational maps the hyperbolic and Hausdorff dimensions of the associated Julia coincide, see \cite[Question 1.1]{rempe}.  This equality is known to hold for many classes of rational maps, for example those satisfying the \emph{topological Collet-Eckmann} (TCE) condition \cite[Theorem 4.3]{TCE}.  However, a recent announcement of Avila and Lyubich based on results from \cite{avila} states that certain Feigenbaum quadratic polynomials yield counter examples.  Similarly, in the more general setting of limit sets of Kleinian groups, if one can find a subset with Hausdorff dimension strictly greater than one which is the limit set of a Schottky group, then our result applies.
\\ \\
We also consider the following variant of the distance set conjecture where one only allows distances realised in a pre-determined set of directions $C \subseteq S^1$, which might be an arc for example.  We define the $C$-\emph{restricted distance set} of $K \subseteq \mathbb{C}$ to be
\[
D_{C}(K) \ = \  \Big\{ \lvert x-y\rvert \, : \, x,y \in K, \, \frac{x-y}{\lvert x-y\rvert} \in C \Big\}.
\]
Clearly one cannot expect the analogue of the distance set conjecture to hold for arbitrary $C$ and $K$.  Indeed, if $K$ is a line segment, then distances are only obtained in one direction and so $D_{C}(K)$ is empty if $C$ does not contain this direction.  However, the CP-chain approach is sufficient to prove the following extension of Theorem \ref{FFSdist}.

\begin{thm} \label{conicaldist}
Let $\mu$ be a measure on $\mathbb{C}$ which generates an ergodic CP-chain and suppose that $K= \text{\emph{supp}}(\mu)$ is not contained in a 1-rectifiable curve and that $\mathcal{H}^1(K) >0$. Then for any $C \subseteq S^1$ with non-empty interior and any conformal map $S$
\[
\dim_{\text{\emph{H}}} D_C\big(S(K) \big) \geq   \min \{ 1, \ \dim_{\text{\emph{H}}} \mu \}.
\]
\end{thm}

We will prove Theorem \ref{conicaldist} in Section \ref{conicaldistproof}.  Following the proof of Theorem \ref{distancesetsconformal}, this yields the following corollary.

\begin{cor}
Consider a conformal system or a system of similarities in the plane and a transitive subshift of finite type $\Sigma_A$ such that $\dim_{\text{\emph{H}}} F_A >1$.  Then $\dim_{\text{\emph{H}}} D_{C} \big(F_A  \big) =1$ for any $C \subseteq S^1$ with non-empty interior.
\end{cor}

\subsubsection{Projections}

In order to obtain results for \emph{all} projections rather than \emph{almost all}, one often needs another assumption guaranteeing a certain homogeneity in the space of projections.  Following \cite{HochmanShmerkin} we now state the version of this extra assumption which we need in our context.
\\ \\
\emph{\textbf{Minimality assumption:} The set $F_A$ corresponding to a subshift of finite type $\Sigma_A$ and a system of similarities satisfies the minimality assumption for $k \in \{1, \dots, d-1\}$ if for all $\pi \in \Pi_{d,k}$ the set
\[
\big\{\pi O(S_{\alpha \vert_k}) \ : \  \alpha  \in \Sigma_A, \ k \in \mathbb{N} \big\}
\]
is dense in $\Pi_{d,k}$, where $O(S_{\alpha \vert_k})$ is the orthogonal part of the map $S_{\alpha \vert_k}$.}
\\ \\
We note that this reduces to the minimality assumption in \cite{HochmanShmerkin} in the case of a full shift however, unlike in the full shift case, there is no useful group action induced on $\Pi_{d,k}$ by the orthogonal parts of the maps in the defining system.

\begin{thm} \label{projectionthm}
Consider a system of similarities and a transitive subshift of finite type $\Sigma_A$ satisfying the strong separation property.  Also assume that $F_A$ satisfies the minimality assumption for some $k <d$ and let $\mu$ be a Gibbs measure for $F_A$. Then for \textbf{all} orthogonal projections $\pi \in \Pi_{d,k}$,
\[
\dim_{\text{\emph{H}}} \pi \mu  \ =  \ \min \big\{ k, \, \dim_{\text{\emph{H}}} \mu \big\}
\]
and
\[
\dim_{\text{\emph{H}}} \pi F_A \ = \ \min \big\{ k, \, \dim_{\text{\emph{H}}} F_A \big\}.
\]
\end{thm}

We will prove Theorem \ref{projectionthm} in Section \ref{projectionthmproof}.  We note that this projection result applies to graph-directed self-similar sets and measures satisfying the strong separation property, thus generalising \cite[Theorem 1.6]{HochmanShmerkin} to the graph-directed setting.
\\ \\
It would be desirable to remove the reliance on the strong separation property from Theorem \ref{projectionthm}.  Concerning dimensions of projections of \emph{measures}, removing the separation property is challenging.  Progress on this problem was made by Falconer and Jin \cite{FalconerJin} and it may be possible to apply their ideas in our setting.  Concerning dimensions of projections of \emph{sets}, the difficulty in removing the separation property is that when one applies Proposition \ref{approxfromwithin}, one cannot guarantee that the subsystem satisfies the minimality assumption even if the original system did.  In the case of self-similar sets modelled by a full shift this was overcome by Farkas \cite{Farkas}.  In $\mathbb{R}^2$ it is straightforward as only one irrational rotation is needed, but in higher dimensions Farkas relied on a careful application of Kronecker's simultaneous approximation theorem.  Extending this approach to our setting may be possible, but the lack of an induced group action could cause problems.  For example, consider a system of similarities in the plane consisting of three maps mapping the unit ball into three pairwise disjoint sub-balls.  Suppose $S_0$ rotates by an irrational angle $\alpha$, $S_1$ rotates by $-\alpha$ and $S_2$ is a homothety.  Now consider the subshift of finite type corresponding to the matrix
 \[
A \ = \ \left ( \begin{array}{ccc}
0 & 1 & 0 \\ 
1 & 0 & 1 \\ 
1 & 0 & 1 \\ 
\end{array} \right ) .
\]
It is easily seen that this subshift is mixing (and so transitive), but that it does not satisfy the minimality condition despite the presence of irrational rotations at the first level.
\\ \\
Obtaining sharpenings of the classical projection theorems for self-conformal sets and measures is more challenging. Theorem \ref{symbtogeomCP} and Theorem \ref{EEE} combine to yield information about the projections of the conformal image $S(\mu_0)$.  In particular, for any $\varepsilon>0$ there is an open dense set of projections which yield dimension within $\varepsilon$ of optimal.  We believe careful applications of this and Theorem \ref{C1images} would allow this to be transferred back to the original measure $\mu$, but we do not include the details.  This would also yield, for example, that the set of optimal projections is residual (contains a dense $G_\delta$ set) in $\Pi_{2,1}$.  The next challenge in proving an `all projections' result is to introduce an appropriate minimality condition. A plausible such condition would be to replace $\pi O(S_{\alpha \vert_k})$ with the (right) action on $\pi$ by the Jacobian derivative $J$ at the fixed point of the map $S_{\alpha \vert_k}$ in our minimality condition stated above.  This certainly reduces to our condition for systems of similarities.  The difficulty is in relating the dimension of $\pi \mu$ with $(\pi J)\mu$ because, unlike in the linear setting, the measures $\pi \mu$ and $(\pi J)\mu \circ S_{\alpha \vert_k}^{-1}$ are not just scaled copies of each other.
\\ \\
In certain situations one can say more.  For example, if the defining system includes a map which is simply an irrational rotation (and contraction) around the origin, then a minimality condition can be satisfied using this map alone without any nonlinear complications.  In some sense this is a very restrictive condition because it relies on the conformal system having a map which is a strict similarity.  However, for our main conformally generated examples in this paper, limit sets of Schottky groups and Julia sets of rational maps, the sets and measures naturally lie in the Riemann sphere and so when shifting to the complex plane (via stereographic projection) we are at liberty to choose which points play the role of $\infty$ and $0$.  For example, if we have a non-parabolic M\"obius map in the defining system, then it is conjugate to a map which fixes zero and $\infty$ and so for a specific choice of coordinates this map becomes a rotation around the origin.   In some sense this discussion indicates that considering orthogonal projections of limit sets of Schottky groups and Julia sets is not particularly natural because of the dependency on the choice of coordinates.

\section{Proofs}

\subsection{An estimate for minimeasures}

In this section we prove a technical lemma which roughly speaking says that  when you restrict a Gibbs measure for a conformal system to a cylinder and normalise, you get a measure (uniformly) equivalent to the obvious conformal image of the appropriate first level restriction.  This will be important when proving Theorem \ref{symbtogeom} in the subsequent section.

\newpage

\begin{lma}\label{gibbsmeasures2}
Consider a conformal system and a subshift of finite type $\Sigma_A$ satisfying the strong separation property and let $\mu$ be a Gibbs measure for $F_A$. Then there exists a uniform constant $C >0$ depending only on the potential $\phi$ such that for all $\alpha \in \Sigma_A$, $m \in \mathbb{N}$ we have
\[
C^{-1} \,\mu\vert_{S_{\alpha\vert_m}(F_A^{\alpha_m})}(E)  \ \leq \  \mu(S_{\alpha\vert_m}(F_A^{\alpha_m})) \, \mu\vert_{F_A^{\alpha_m}} \circ S_{\alpha\vert_m}^{-1}(E)  \ \leq \   C \,\mu\vert_{S_{\alpha\vert_m}(F_A^{\alpha_m})}(E)
\]
for all Borel sets $E \subseteq \mathbb{C}$.
\end{lma}

\begin{proof}
Since the strong separation property is satisfied, it suffices to prove the symbolic version of the lemma.  Since the cylinders generate the Borel sets in $\Sigma_A$, it suffices to show that for $\alpha \in \Sigma_A$, $m \in \mathbb{N}$, $\beta \in \Sigma$ such that $\alpha_{m+1}\beta \in \Sigma_A$, and $n \in \mathbb{N}$ the quantity
\[
\frac{\mu\big([\alpha \vert_{m+1} \beta\vert_n]\big)}{\mu\big([\alpha \vert_{m+1} ]\big)\mu\big([ \beta\vert_n]\big)}
\]
is uniformly bounded away from zero and $\infty$ independently of $\alpha$, $m$, $\beta$ and $n$. This follows from a more or less standard Gibbs measure argument, but we include the details for completeness.  By (\ref{gibbsbowen}), we have
\begin{eqnarray*}
\frac{\mu\big([\alpha \vert_{m+1} \beta\vert_n]\big)}{\mu\big([\alpha \vert_{m+1} ]\big)\mu\big([ \beta\vert_n]\big)}
& \leq & \frac{C_2 \, \exp \Big(\phi^{m+n+1}\big(\alpha \vert_{m+1} \beta \big) \, - \, (m+n+1)P(\phi) \Big)}{C_1 \, \exp \Big(\phi^{m+1}(\alpha) \, - \, (m+1)P(\phi) \Big) \ C_1 \, \exp \Big( \phi^{n}(\beta) \, - \, nP(\phi) \Big)} \\ \\
& = & \frac{C_2}{C_1^2} \ \exp \Bigg(\sum_{l=0}^{m+n} \phi\big(\sigma^{l}(\alpha \vert_{m+1} \beta) \big) - \sum_{l=0}^{m} \phi\big(\sigma^{l}(\alpha) \big) - \sum_{l=0}^{n-1}\phi\big(\sigma^{l}(\beta) \big)\Bigg) \\ \\
& = & \frac{C_2}{C_1^2} \ \exp \Bigg(\sum_{l=0}^{m} \phi\big(\sigma^{l}(\alpha \vert_{m+1} \beta) \big) - \sum_{l=0}^{m} \phi\big(\sigma^{l}(\alpha) \big)\Bigg) \\ \\
& \leq &  \frac{C_2}{C_1^2} \ \exp \Bigg(\sum_{l=0}^{m}  \bigg\lvert  \phi\big(\sigma^{l}(\alpha \vert_{m+1} \beta) \big)  \, - \, \phi\big(\sigma^{l}(\alpha) \big) \bigg\rvert \Bigg) \\ \\
& \leq &  \frac{C_2}{C_1^2} \ \exp \Bigg(\sum_{l=0}^{m} \text{var}_{m+1-l}(\phi)\Bigg) \\ \\
& \leq &  \frac{C_2}{C_1^2} \ \exp \Bigg(\sum_{l=0}^{\infty} \text{var}_{l}(\phi)\Bigg) \\ \\
& <& \infty.
\end{eqnarray*}
This establishes the required upper bound.  Since $\alpha_{m+1}\beta \in \Sigma_A$, a similar argument yields the required lower bound  
\[
\frac{\mu\big([\alpha \vert_{m+1} \beta\vert_n]\big)}{\mu\big([\alpha \vert_{m+1} ]\big)\mu\big([ \beta\vert_n]\big)} \  \geq  \  \frac{C_1}{C_2^2} \ \exp \Bigg( - \sum_{l=0}^{\infty} \text{var}_{l}(\phi)\Bigg)  \ > \ 0,
\]
which completes the proof.
\end{proof}

\subsection{Proof of Theorem \ref{symbtogeom}}  \label{symbtogeomproof}

Consider a conformal system and a subshift of finite type $\Sigma_A$ satisfying the strong separation property and let $\mu$ be a Gibbs measure for $F_A$.  Let $\nu \in  \text{Mini}(\mu)$ which we assume is not supported on the boundary of $U$.  We will assume for simplicity that we are using the 2-adic partition operator.  The general $b$-adic case is similar.  Choose $x \in \text{supp}(\nu) \cap U$ and let
\[
r \  = \  \inf_{y \in  \partial U} \lvert x-y \rvert \ > \ 0.
\]
Since $\nu$ is a minimeasure, there exists some (closed) dyadic square $B$ with sidelengths $2^{-k}$ for some $k \in \mathbb{N}$ for which
\[
\nu \ = \  \mu^B \  = \  \frac{1}{\mu(B)} \, \mu\lvert_B \, \circ \, T_B^{-1}.
\]
Recall that $T_B$ is the unique orientation preserving onto similarity mapping $B$ to $\overline{U}$. Observe that $B(T_B^{-1}(x), 2^{-k} r) \subseteq B$ and that $T_B^{-1}(x) = \Pi(\alpha) \in F_A$ for some $\alpha \in \Sigma_A$.  Choose the unique $m \in \mathbb{N}$ satisfying
\[
\text{Lip}^+\big(S_{\alpha\vert_m}\big) \, < \,  2^{-k}r/\sqrt{2} 
\]
but
\[
\text{Lip}^+\big(S_{\alpha\vert_{m-1}}\big) \, \geq \, 2^{-k} r/\sqrt{2}
\]
and note that
\[
S_{\alpha\vert_{m}}\big( \overline{U} \big) \, \subseteq  \, B(T_B^{-1}(x),2^{-k} r) \, \subseteq \, B.
\]
The bounded distortion property gives
\[
\text{Lip}^-\big(S_{\alpha\vert_m}\big) \ \geq \   \text{Lip}^+\big(S_{\alpha\vert_m}\big) / L \ \geq \  \frac{2^{-k} r}{L \sqrt{2}}.
\] 
Let $S = T_B \circ S_{\alpha\vert_m}$ and $i = \alpha_m \in \mathcal{I}$.  Observe that $S\big(F_A^i\big) \subseteq \text{supp}(\nu)$ and, moreover,  Lemma \ref{gibbsmeasures2} implies
\begin{eqnarray*}
\nu\vert_{S(F_A^i)} &=& \frac{1}{\mu(B)} \, \mu \vert_{S_{\alpha\vert_m}(F_A^i)} \circ T_B^{-1} \\ \\
&\equiv & \frac{\mu\big(S_{\alpha\vert_m} (F_A^i) \big)}{\mu(B)} \  \mu \vert_{F_A^i} \circ S_{\alpha\vert_m} ^{-1} \circ T_B^{-1} \\ \\
&\equiv &  \mu_i \circ S^{-1}.
\end{eqnarray*}
This combined with the fact that $S$ is conformal proves the required result for minimeasures. We will now turn to the proof for micromeasures $\nu$, which is conceptually more difficult but can be circumvented by a compactness argument.  Note that for the construction above
\begin{equation} \label{bddaway0}
\frac{ r}{L \sqrt{2}} \ \leq \   \text{Lip}^-\big(S\big) \ \leq \ \text{Lip}^+\big(S\big) \ \leq \  \frac{ r}{\sqrt{2}}
\end{equation}
and
\begin{equation} \label{bddaway}
\frac{\mu\big(S_{\alpha\vert_m} (F_A^i) \big)}{\mu(B)}  \ \geq \ p(r) \ > \ 0
\end{equation}
for some uniform weight $p(r)>0$ depending only on $r$.  This can be seen since $\mu$ is a Gibbs measure for a conformal system satisfying the strong separation condition and so is doubling.  Let $\nu \in  \text{Micro}(\mu)$ which we assume is not supported on the boundary of $U$. Choose a point $z \in \text{supp}(\nu) \cap U$ and let
\[
r(z) \  = \  \frac{1}{2} \, \inf_{y \in  \partial U} \lvert z-y \rvert \ > \ 0.
\]
It follows that there exists a sequence of points $x_l \in \overline{U}$ and a sequence of minimeasures $\nu_l \in \text{Mini}(\mu)$ satisfying $x_l \in B(z,r(z)) \cap \text{supp}(\nu_l)$ for all $l \in \mathbb{N}$, $x_l \to z$ and $\nu_l \to_{w^*} \nu$.  Repeat the above argument for each minimeasure $\nu_l$, observing that we can choose $x=x_l$ and then for each $l$, the value $r$ is at least $r(z)$.  This means we can find a sequence $\{(S_l, i_l,\mu_{0,l}) \}_{l \in \mathbb{N}}$ satisfying the following properties:
\begin{itemize}
\item[(1)] the $S_l$ are conformal maps on $U$ and by (\ref{bddaway0}) the Lipschitz constants of the $S_l$ are bounded uniformly away from 0 and 1 independent of $l$.
 \item[(2)] for each $l \in \mathbb{N}$, $i_l \in \mathcal{I}$.
 \item[(3)] the $\mu_{0,l}$ are probability measures and for each $l \in \mathbb{N}$, $\mu_{0,l}$ is equivalent to $\mu_{i_l}$ with Radon-Nikodym derivative bounded uniformly away from $0$ and $\infty$ independent of $l$.  This independence from $l$ comes from Lemma \ref{gibbsmeasures2} and (\ref{bddaway}).
\item[(4)] for each $l \in \mathbb{N}$, $\nu_l\vert_{S_l(F_A^{i_l})} \ = \ \mu_{0,l} \circ S_l^{-1}$.
\end{itemize}
A combination of Tychonoff's Theorem, The Arzel\'a-Ascoli Theorem and Prokhorov's Theorem implies that we may extract a subsequence of the triples $(S_l, i_l,\mu_{0,l})$ such that the $S_l$s converge uniformly, the $i_l$s eventually become constantly equal to some $i \in \mathcal{I}$ and the $\mu_{0,l}$s converge weakly to a Borel measure $\mu_{0}$ equivalent to $\mu_{i}$.  Since uniform limits of complex analytic maps are complex analytic and the uniform bounds on the Lipschitz constants of the $S_l$ guarantees that the uniform limit has non-vanishing derivative, the limit $S$ is conformal.  Recall that a map is conformal on an open domain if and only if it is holomorphic (equivalently analytic) and its derivative is everywhere non-zero on its domain.  Since $\nu_l \to_{w^*} \nu$, it follows that
\[
\nu\vert_{S(F_A^{i})} \ = \ \mu_{0} \circ S^{-1}
\]
which completes the proof.

\subsection{Proof of  Theorem \ref{symbtogeomCP}}  \label{symbtogeomCPproof}

Consider a conformal system and a subshift of finite type $\Sigma_A$ satisfying the strong separation property and let $\mu$ be a Gibbs measure for $F_A$. Theorem 7.10 in \cite{HochmanShmerkin} implies that there exists an ergodic CP-chain for $\mu$ supported on $ \text{Micro}(\mu)$ of dimension at least $\dim_\H \mu$. Writing $Q$ for the measure component, Theorem 7.7 in \cite{HochmanShmerkin} tells us that $Q$-almost all $\nu \in \text{Micro}(\mu)$ generate this CP-chain.  We now wish to apply Theorem \ref{symbtogeom} to a typical micromeasure, but we do not know \emph{a priori} that typical micromeasures are not supported on the boundary of the cube.  However, a solution to this problem can be found by examining the proof of Theorem 7.10 in \cite{HochmanShmerkin}.  Indeed, by applying a random homothety which does not effect the set of micromeasures one is able to argue that almost surely (with respect to the randomisation of the homothety and the measure component of the resultant CP-chain) the micromeasures satisfy $\nu\big( U \big) = 1$.
 Thus we may fix a distinguished micromeasure $\nu \in \text{Micro}(\mu)$, not supported on the boundary of $\overline{U}$ and which generates the same ergodic CP-chain with measure component $Q$.  Theorem \ref{symbtogeom} now tells us that there exists a conformal map $S$ on $U$ and a measure $\mu_0 \equiv \mu_i$ for some $i \in \mathcal{I}$, such that
\[
\nu\vert_{S(F_A)} \ = \  \mu_0 \circ S^{-1}.
\]
Lemma 7.3 in \cite{HochmanShmerkin} implies that (the normalisation of) this measure generates $Q$, which proves the result.

\subsection{Proof of Proposition \ref{approxfromwithin}}  \label{approxfromwithinproof}

Consider a conformal system or a system of similarities and let $\Sigma_A$ be a transitive subshift of finite type.  Write $X$ for the compact metric space the system of maps acts on and let
\[
\mathcal{A}_0\big(\Sigma_A, k\big) \ = \ \big\{ \alpha\vert_k: \alpha \in [0]\cap \Sigma_A \big\}
\]
be the set of admissible words of length $k$ in $\Sigma_A$ which begin with the symbol 0. Let
\[
s = \dim_\H  F_A^0 = \dim_\H F_A 
\]
For $\textbf{\emph{i}} \in \mathcal{I}^k$, choose a ball $B_{\textbf{\emph{i}},k}$ of radius $\text{Lip}^+(S_\textbf{\emph{i}}) \lvert X \rvert$ containing $S_\textbf{\emph{i}}(X)$.  Clearly the collection
\[
\{B_{\textbf{\emph{i}},k}\}_{\textbf{\emph{i}} \in \mathcal{A}_0(\Sigma_A, k)}
\]
forms a cover of $F_A^0$ for all $k \in \mathbb{N}$.  For each $k \in \mathbb{N}$, use the Vitali Covering Lemma to find a subset $\mathcal{D}_0\big(\Sigma_A, k\big) \subseteq \mathcal{A}_0\big(\Sigma_A, k\big)$ such that
\[
\{B_{\textbf{\emph{i}},k}\}_{\textbf{\emph{i}} \in\mathcal{D}_0(\Sigma_A, k)}
\]
is a pairwise disjoint collection of balls and such that
\[
F_A^0 \ \subseteq \ \bigcup_{\textbf{\emph{i}} \in\mathcal{D}_0(\Sigma_A, k)} \, 3 B_{\textbf{\emph{i}},k}
\]
where $3 B_{\textbf{\emph{i}},k}$ is the ball centered at the same point as $B_{\textbf{\emph{i}},k}$, but with three times the radius. Fix $\tau \in \mathcal{I}$ such that $\tau 0$ is an admissible word and for each $\textbf{\emph{i}} \in \mathcal{D}_0\big(\Sigma_A, k\big)$, let $\textbf{\emph{j}}$ be a finite word of minimal length such that $\textbf{\emph{i}}\textbf{\emph{j}}$ is both admissible and ends in the symbol $\tau$. Such a word $\textbf{\emph{j}}$ exists since $\Sigma_A$ is transitive and if more than one choice for $\textbf{\emph{j}}$ exists, then choose one arbitrarily.  Observe that there exists a universal bound $K \in \mathbb{N}$ such that $\lvert \textbf{\emph{j}}\rvert \leq K$ for all $\textbf{\emph{i}}$ and $k$.   Let
\[
\mathcal{D}_0^\tau\big(\Sigma_A, k\big) \ = \ \big\{ \textbf{\emph{i}}\textbf{\emph{j}}: \textbf{\emph{i}} \in \mathcal{D}_0\big(\Sigma_A, k\big) \big\}.
\]
It follows that there is a constant $C>1$ depending only on $K$ and the maps in the original system such that
\begin{equation} \label{givescover}
F_A^0 \ \subseteq \ \bigcup_{\textbf{\emph{i}} \in \mathcal{D}_0^\tau(\Sigma_A, k)} \, C B_{\textbf{\emph{i}},k}
\end{equation}
where $C B_{\textbf{\emph{i}},k}$ is the ball centered at the same point as $B_{\textbf{\emph{i}},k}$, but with radius multiplied by $C$.  Let $\Sigma_k = \Sigma\big(\mathcal{D}_0^\tau\big(\Sigma_A, k\big)\big)$ be the full shift over the alphabet $\mathcal{D}_0^\tau\big(\Sigma_A, k\big)$ which clearly satisfies the strong separation property.  Since $\Sigma_k$ is a \emph{full shift}, the value $t(k)$ defined uniquely by
\[
\sum_{\textbf{\emph{i}} \in  \mathcal{D}_0^\tau(\Sigma_A, k)} \text{Lip}^-(S_\textbf{\emph{i}})^{t(k)} \ = \ 1
\]
is a lower bound for the Hausdorff dimension of $\Pi(\Sigma_k) $, see \cite[Proposition 9.7]{Falconer}.  Let $\varepsilon>0$ and observe that, using the increasingly fine covers of $F_A^0$ given by (\ref{givescover}), we have
\begin{eqnarray*}
\infty \ = \ \mathcal{H}^{s-\varepsilon}\big(F_A^0\big) & \leq &  \lim_{k \to \infty} \sum_{\textbf{\emph{i}} \in\mathcal{D}_0^\tau(\Sigma_A, k)} \Big(2C \,  \text{Lip}^+(S_\textbf{\emph{i}}) \lvert X \rvert \Big)^{s-\varepsilon} \\ \\
& \leq &  (2CL\lvert X \rvert)^{s-\varepsilon} \, \lim_{k \to \infty} \sum_{\textbf{\emph{i}} \in\mathcal{D}_0^\tau(\Sigma_A, k)}  \text{Lip}^-(S_\textbf{\emph{i}})^{s-\varepsilon}.
\end{eqnarray*}
Hence we may choose $k \in \mathbb{N}$ large enough to guarantee that
\[
\sum_{\textbf{\emph{i}} \in\mathcal{D}_0^\tau(\Sigma_A, k)}  \text{Lip}^-(S_\textbf{\emph{i}})^{s-\varepsilon} >1
\]
which implies that $s-\varepsilon <  t(k) \leq  \dim_\H \Pi(\Sigma_k)$ and letting $\Sigma_\varepsilon = \Sigma_k$ and observing that $\Pi(\Sigma_\varepsilon) \subseteq F_A$ completes the proof.

\subsection{Proof of Theorem \ref{FFS+dist}}    \label{FFS+distproof}

Let $\mu$ be a probability measure supported on a compact set $F \subseteq \mathbb{C}$ which generates an ergodic CP-chain, suppose  $\mathcal{H}^1(F) >0$ and let $S$ be a conformal map. The \emph{direction set} of a set $K \subseteq \mathbb{C}$ is defined by
\[
\text{dir}(K) \ = \ \bigg\{\frac{x-y}{\lvert x-y\rvert} : x,y \in K, \, x \neq y \bigg\} \  \subseteq \  S^1.
\]
Orponen \cite{Orponen} observed that a set $K$ with $\mathcal{H}^1(K) >0$ is either contained in a rectifiable curve or has a dense direction set.  If $F$ is contained in a rectifiable curve, then so is $S(F)$.  This, combined with the fact that $\mathcal{H}^1\big(S(F)\big) >0$, implies that $D(S(F))$ contains an interval by a result of Besicovitch and Miller \cite{besmil}, completing the proof in this case.  Now suppose that $\text{dir}(F)$ is dense in $S^1$.  Let $\varepsilon > 0$ and choose $\pi \in \Pi_{2,1}$ which satisfies
\begin{equation}\label{firsteps}
\dim_\H \pi \mu > \min\{ 1, \dim_\H \mu \} - \varepsilon.
\end{equation}
The existence of such a $\pi$ is guaranteed by Theorem \ref{EEE} for example.  Also, let $\delta>0$ depending on $\varepsilon$ be the value given to us by Theorem \ref{C1images} and let $\delta'>0$ be chosen depending on $\delta$. Since $S$ is conformal we may find a point $x_0 \in F$ and $R>0$ sufficiently small, so that
\[
\| S\vert_{B(x_0,R)} - D_{x_0}S \| \leq \delta',
\]
\[
\mathcal{H}^1\big(B(x_0,R) \cap F\big) >0
\]
and $B(x_0,R) \cap F$ is not contained in a rectifiable curve.  It follows from Orponen's observation that $\text{dir}\big(B(x_0,R) \cap F\big)$ is dense in $S^1$.  Thus, identifying $\Pi_{2,1}$ with $S^1$ in the natural way, we may choose two points $x,y \in B(x_0,R/2) \cap F$ such that the direction
\[
\frac{x-y}{\lvert x-y\rvert} \in S^1
\]
determined by $x$ and $y$ is $\delta'$ close to $\pi$.  Now choose $r \in (0,\lvert x-y\rvert/3)$ sufficiently small to ensure that for all $z \in B(y,r)$ we have
\[
\bigg\lvert \frac{1}{\lvert D_{x_0}S \rvert} \lvert S(x) - S(z) \rvert  \, - \, \lvert \pi(z) - \pi(x) \rvert \bigg\rvert \leq \delta'.
\]
Now define a map $g_x: \overline{U} \setminus B(x, \lvert x-y\rvert/3)  \to \mathbb{R}$ by
\[
g_x(z) = \frac{1}{\lvert D_{x_0}S \rvert} \lvert S(z)-S(x) \rvert.
\]
Notice that $g_x$ is $C^1$ and can be extended to a $C^1$ mapping $g$ on the whole of $ \overline{U}$. Since $\delta'$ was chosen to depend on $\delta$, it is readily seen that it can be chosen small enough to guarantee that the derivative of $g$ is sufficiently close to $\pi$ on $B(y,r)$, i.e.
\begin{equation}\label{forg}
\sup_{z \in B(y,r)} \| D_z g - \pi\| < \delta.
\end{equation}
Consider the restricted and normalised measure 
$$\nu = \mu(B(y,r))^{-1}\mu\vert_{B(y,r)}.$$  
It is a consequence of the Besicovitch Density Point Theorem that $\nu$ generates the \textit{same} CP-chain as $\mu$; see for example \cite[Lemma 7.3]{HochmanShmerkin}.   Theorem \ref{C1images} combined with (\ref{forg}) gives us that
\begin{equation}\label{secondeps}
\dim_\H g \nu \geq  \dim_\H \pi \mu - \varepsilon.
\end{equation}
Since $g$ maps $B(y,r) \cap F$ into
\[
 \frac{1}{\lvert D_{x_0}S \rvert} D(S(F)),
\]
we have that $g\nu$ is supported on this set and so
\begin{eqnarray*}
\dim_\H  D(S(F)) \ = \  \dim_\H  \frac{1}{\lvert D_{x_0}S \rvert} D(S(F)) &\geq& \dim_\H g  \nu  \\
&\geq& \dim_\H \pi \mu - \varepsilon  \quad \text{by (\ref{secondeps})} \\
&\geq&\min\{1,\dim_\H \mu\}-2\varepsilon \quad \text{by (\ref{firsteps})}
\end{eqnarray*}
which proves the result since $\varepsilon > 0$ was arbitrary.

\subsection{Proof of Theorem \ref{distancesetsconformal}}    \label{distancesetsconformalproof}

Consider a conformal system or a system of similarities in the plane and let $\Sigma_A$ be a transitive subshift of finite type with $\dim_{\text{{H}}} F_A >1$. Proposition \ref{approxfromwithin} implies that there exists a full shift $\Sigma_0$ over a potentially different (but finite) alphabet made up of a finite collection of restrictions of elements in $\Sigma_A$ such that $\Pi(\Sigma_0) \subseteq F_A$,   $\dim_{\text{{H}}} \Pi(\Sigma_0)>1$ and such that the conformal system corresponding to $\Sigma_0$ satisfies the strong separation property.  Let $\mu$ be a Gibbs measure supported on $\Pi(\Sigma_0)$ with $\dim_\H \mu = \dim_\H \Pi(\Sigma_0)$, which exists by, for example, \cite{GatzourasPeres}.  Theorem \ref{symbtogeomCP} guarantees that there exists a conformal map $S$ and a probability measure $\mu_0 \equiv \mu_i$ for some $i \in \mathcal{I}$, such that $\mu_0 \circ S^{-1}$ generates an ergodic CP-chain. Now since $S^{-1}$ is also a conformal map, Theorem \ref{FFS+dist} implies that
\[
\dim_{\text{{H}}} D\big(  S^{-1}\big(\text{{supp}}(\mu_0 \circ S^{-1}) \big) \big) \, \geq \,  \min \{ 1, \ \dim_{\text{{H}}} \mu_0 \circ S^{-1} \}  \, = \,   \min \{ 1, \ \dim_\H \mu \}  \, =  \, 1.
\]
However,
\[
S^{-1}\big(\text{{supp}}(\mu_0 \circ S^{-1})\big) \, = \,  S^{-1} \big(S\big(\text{{supp}}(\mu_0)\big) \big) \,  \subseteq \,  \text{{supp}}(\mu) \, = \,  \Pi(\Sigma_0) \subseteq F_A
\]
since  $\mu_0 \equiv \mu_i$, which proves that $\dim_{\H} D \big(F_A  \big) =1$ as required.  If we assume the strong separation condition, then $\dim_{\text{{H}}} F_A \geq 1$ is sufficient because we do not need to approximate from within and such sets satisfy $\mathcal{H}^1(F_A )>0$.

\subsection{Proof of Theorem \ref{conicaldist}}    \label{conicaldistproof}

This proof is similar to either the proof of Theorem \ref{distancesetsconformal} given in the Section \ref{FFS+distproof}, or the proof of  \cite[Theorem 1.7]{FergusonFraserSahlsten}, and so the details are omitted.  In the proof of Theorem \ref{distancesetsconformal}, a distinguished $\pi$ is chosen and choosing $r>0$ small enough guarantees that all of the distances considered are realised by directions arbitrarily close to $\pi$.  Since $C$ is assumed to have nonempty interior, the direction set $\text{dir}(\text{supp}(\mu)) $ is dense in $S^1$, and the set $\Pi_\varepsilon$ of `$\varepsilon$-good' projections is open, dense and of full measure, we may choose the distinguished $\pi$ to lie in the intersection of $C$ and $\Pi_\varepsilon$ and find a direction in $\text{dir}(\text{supp}(\mu)) $ in $C$ which is $\delta$ close to this $\pi$.  Provided we choose $r$ small enough to ensure that all directions realised by points in $x$ together with points in $B(y,r)$ also lie in $C$, the rest of the proof proceeds as the proof of Theorem \ref{distancesetsconformal}.

\subsection{Proof of Theorem \ref{projectionthm}} \label{projectionthmproof}

Consider a system of similarities acting on $[0,1]^d$ and a transitive subshift of finite type $\Sigma_A$ satisfying the strong separation property.  Also assume that $F_A$ satisfies the minimality assumption for some $k <d$ and let $\mu$ be a Gibbs measure for $F_A$.  Theorem \ref{symbtogeomCP2} guarantees that there exists a similarity map $S$ on $[0,1]^d$ and a probability measure $\mu_0 \equiv \mu_i$ for some $i \in \mathcal{I}$, such that $ \nu = \mu_0 \circ S^{-1}$ generates an ergodic CP-chain.  Let $\varepsilon>0$ and observe that Theorem \ref{EEE} implies that there is an open dense set $\mathcal{V}_\varepsilon \subseteq \Pi_{d,k}$ such that for $\pi \in \mathcal{V}_\varepsilon$
\[
\dim_\H \pi \nu \, > \,  \min\{ k, \dim_\H \nu\} - \varepsilon.
\]
However, since $\nu = \mu_0 \circ S^{-1}$, the measures $\big(\pi O(S)^{-1}\big) \mu_0$ and $\pi \nu$ are essentially the same measure (one is equivalent to a scaled and translated copy of the other).  Therefore, if $\pi \in \mathcal{V}_\varepsilon$, then
\[
\dim_\H \big( \pi O(S)^{-1}\big) \mu_0   \, =  \, \dim_\H \pi \nu  \,  > \,  \min\{ k, \dim_\H \nu\} - \varepsilon  \, = \,  \min\{ k, \dim_\H \mu_0\}- \varepsilon .
\]
Since $O(S)^{-1}$ acts as an isometry on $\Pi_{d,k}$, this gives that $\mathcal{U}_\varepsilon :=  \mathcal{V}_\varepsilon O(S)^{-1} \subseteq \Pi_{d,k}$ is open, dense and if $\pi \in \mathcal{U}_\varepsilon$, then
\[
\dim_\H \pi \mu_0  \, > \,  \min\{ k, \dim_\H \mu_0\}- \varepsilon  \, = \,  \min\{ k, \dim_\H \mu\}- \varepsilon.
\]
Of course, we really want this estimate for the original measure $\mu$ but this can be achieved by a simple trick. Since $\mu_0$ is equivalent to $\mu_i$ and since $\Sigma_A$ is transitive,  for all $j \in \mathcal{I}$, we can find a finite word $\textbf{\emph{i}}' \in \mathcal{I}^k$, beginning with $i$ and such that $\textbf{\emph{i}}'j$ is admissible, which satisfies $S_{\textbf{\emph{i}}'}(F_A^j) \subseteq \text{supp} \mu_0$.  Crucially, when $\mu_0$ is restricted to this subset it is equivalent to $\mu_j \circ S_{\textbf{\emph{i}}'}^{-1}$.  The Besicovitch Density Point Theorem guarantees that (the normalisation of) the push forward under $S$ of this restriction of $\mu_0$ generates the same ergodic CP-chain as $\nu$.  This means that the `good set' $\mathcal{V}_\varepsilon$ also applies to $\mu_j \circ S_{\alpha_k}^{-1} \circ S^{-1}$ and, using a similar argument to above, the set $\mathcal{U}^j_\varepsilon :=  \mathcal{U}_\varepsilon O(S \circ S_{\textbf{\emph{i}}'})^{-1} \subseteq \Pi_{d,k}$  is a `good set' for the measure $\mu_j$, i.e, for $\pi \in \mathcal{U}^j_\varepsilon$,
\[
\dim_\H \pi \mu_j  \, >  \, \min\{ k, \dim_\H \mu\}- \varepsilon.
\]
Letting
\[
\mathcal{U}^{all}_\varepsilon = \bigcap_{j \in \mathcal{I}} \mathcal{U}^j_\varepsilon
\]
we obtain an open and dense set which is good for all first level measures $\mu_i$ simultaneously.  It follows that if $\pi \in \mathcal{U}^{all}_\varepsilon $, then
\[
\dim_\H \big(\pi O(S_{\alpha\vert_k}) \big) \mu_i   \, =  \,  \dim_\H \pi  \mu_i \, > \,  \min\{ k, \dim_\H \mu\}- \varepsilon
\]
for all $\alpha \in \Sigma_A$, $k \in \mathbb{N}$ and $i \in \mathcal{I}$.  The minimality condition combined with the openness of $\mathcal{U}^{all}_\varepsilon$ now implies that
\[
\dim_\H \pi\mu_i  \, > \,  \min\{ k, \dim_\H \mu\}- \varepsilon
\]
for \emph{all} $\pi \in \Pi_{d,k}$ and $i \in \mathcal{I}$.  Finally, to transfer this result to $\mu$, observe that if $E$ is such that $\pi\mu(E) >0$, then there must exist $i \in \mathcal{I}$ such that $\pi\mu_i(E) >0$ and therefore
\[
\dim_\H \pi \mu  \, \geq  \, \min_{i \in \mathcal{I}} \dim_\H \pi \mu_i  \, >  \,  \min\{ k, \dim_\H \mu\}- \varepsilon
\]
for \emph{all} $\pi \in \Pi_{d,k}$.  The result now follows since $\varepsilon>0$ was arbitrary.  The final part of the theorem, concerning the dimensions of projections of the set $F_A$, follows easily from the result concerning measures.

\vspace{8mm}

\begin{centering}

\textbf{Acknowledgements}

The majority of this research was carried out while JMF was a PDRA of MP at the University of Warwick.  JMF and MP were financially supported in part by the EPSRC grant EP/J013560/1.

\end{centering}

\end{document}